\documentclass[reqno,12pt]{amsart}
\usepackage{geometry}
\usepackage{mathtools,amssymb,amsthm,mathrsfs,color,lineno,paralist,graphicx,float}
\usepackage{amsfonts}
\usepackage{constants}

\DeclareMathOperator\supp{supp}
\numberwithin{equation}{section}
\usepackage[colorlinks,
linkcolor=red,
anchorcolor=green,
citecolor=blue, 
]{hyperref}

\usepackage{enumitem}
\usepackage[T1]{fontenc}
\usepackage[utf8]{inputenc}

\usepackage{calc}
\definecolor{bleu1}{RGB}{0,57,128}
\def\bleu1{\color{bleu1}}

\usepackage{etoolbox}
\patchcmd{\section}{\normalfont}{\normalfont \bleu1}{}{}
\patchcmd{\subsection}{\normalfont}{\normalfont \bleu1}{}{}
\patchcmd{\subsubsection}{\normalfont}{\normalfont \bleu1}{}{}

\usepackage{xpatch}
\xpatchcmd{\proof}{\itshape}{\it \bleu1 \proofnamefont}{}{}
\newcommand{\proofnamefont}{}

\newtheorem{proposition}{Proposition}[section]
\newtheorem{theorem}{Theorem}[section]
 
\newtheorem{lemma}{Lemma}[section]
\newtheorem{remark}{Remark}[section]
\newtheorem{corollary}{Corollary}[section]

\newcommand{\Z}{{\mathbb Z}}
\newcommand{\R}{{\mathbb R}}
\newcommand{\Q}{{\mathbb Q}}

\usepackage{tikz,tikz-3dplot}
\usepackage{pgfplots}

\makeatletter
\newsavebox{\@brx}
\newcommand{\llangle}[1][]{\savebox{\@brx}{\(\m@th{#1\langle}\)}%
\mathopen{\copy\@brx\kern-0.5\wd\@brx\usebox{\@brx}}}
\newcommand{\rrangle}[1][]{\savebox{\@brx}{\(\m@th{#1\rangle}\)}%
\mathclose{\copy\@brx\kern-0.5\wd\@brx\usebox{\@brx}}}
\makeatother

\begin{document}

\title[]{Spectrum of  Hatano-Nelson model with strictly ergodic potentials}

\author{Xueyin Wang}
\address{
Chern Institute of Mathematics and LPMC, Nankai University, Tianjin 300071, China
}

\email{xueyinwang@mail.nankai.edu.cn}

\author{Zhenfu Wang}
\address{
Chern Institute of Mathematics and LPMC, Nankai University, Tianjin 300071, China
}

\email{zhenfuwang@mail.nankai.edu.cn}

\author {Jiangong You}
\address{
Chern Institute of Mathematics and LPMC, Nankai University, Tianjin 300071, China} \email{jyou@nankai.edu.cn}

\author{Qi Zhou}
\address{
Chern Institute of Mathematics and LPMC, Nankai University, Tianjin 300071, China
}

\email{qizhou@nankai.edu.cn}

\begin{abstract}
We provide a precise formula for the spectrum of the Hatano-Nelson model with strictly ergodic potentials in terms of its Lyapunov exponent. As applications, one clearly observes the real-complex spectrum transition. Moreover, if the Lyapunov exponent is continuous, the spectrum of the Hatano-Nelson model in $\ell^{2}(\mathbb{Z})$ can be approximated by the spectrum of its finite-interval truncation with periodic boundary conditions.  Both of these results are strikingly different from the Hatano-Nelson model with random potentials \cite{Dav01A, Dav01, Dav02}.
\end{abstract}

\maketitle

\section{Introduction}
The Hatano-Nelson model, denoted as $\mathcal{H}(g):\ell^{2}(\mathbb{Z})\rightarrow \ell^{2}(\mathbb{Z})$, is defined as follows:
\begin{equation}\label{HNV}
[\mathcal{H}(g)u]_{n}= -\mathrm{e}^{g}u_{n+1}-\mathrm{e}^{-g}u_{n-1}+ V(n)u_{n},
\end{equation}
where $g\in\mathbb{R}$ is referred to as the nonreciprocal hopping parameter, and $V(\cdot):\mathbb{Z}\rightarrow\mathbb{C}$ is the potential. We only consider the case $g> 0$, as the case $g<0$ is similar. The Hatano-Nelson model was first proposed by Hatano and Nelson \cite{HN96} in 1996, motivated by the study of superconductivity. This model describes the behavior of a quantum particle on a one-dimensional tight-binding lattice with a superimposed imaginary magnetic flux. It was later found that this model plays a crucial role in population biology \cite{AHN,NS98} and solid-state physics \cite{HN97, HN98}. Recently, the Hatano-Nelson model has once again gained attention due to the topological phase transition in non-Hermitian physics \cite{Gong}.

On the other hand, the Hatano-Nelson model naturally appears in the study of  quasi-periodic Schr\"odinger operators. 
Recent quantitative global theory \cite{Av0, GJYZ} of one-frequency  analytic quasi-periodic Schr\"odinger operators actually  establishes a Thouless type formula for the  Schr\"odinger operator
\begin{equation}\label{SO}
[H(x+\mathrm{i}g)u]_{n}=-u_{n+1}-u_{n-1}+v(x+ n\alpha+\mathrm{i}g)u_{n}
\end{equation}
and its dual operator (via Aubry duality)
\begin{equation}\label{dual}
[\widehat{L}(g,x)u]_{n}=-\sum_{k\in\mathbb{Z}}\widehat{v}_{k}\mathrm{e}^{2\pi kg}u_{n+k}+2\cos2\pi (x+ n\alpha)u_{n},
\end{equation} 
where $x\in\mathbb{T}:=\mathbb{R}/\mathbb{Z}$, $v\in C^{\omega}(\mathbb{T},\mathbb{R})$, and $\alpha\in\mathbb{R}\backslash\mathbb{Q}$.
In particular, if we let $v(x)=2\cos2\pi(x)$, then \eqref{SO} is the complexified almost Mathieu operator and its dual operator \eqref{dual} is just the Hatano-Nelson model \eqref{HNV} with the quasi-periodic potential $V(n)=2\cos2\pi(x+n\alpha)$. 

In this paper, we consider the dynamical defined Hatano-Nelson model, denoted by $\{\mathcal{H}(g,x)\}_{x\in X}$, which is described as:
\begin{equation}\label{ddhn}
[\mathcal{H}(g,x)u]_{n}= -\mathrm{e}^{g}u_{n+1}-\mathrm{e}^{-g}u_{n-1}+ v(T^{n}x)u_{n},
\end{equation}
where $(X,\mathcal{B},\mu)$ is a probability measure space and $T: X\rightarrow X$ is an invertible measure-preserving transformation.  Denote by $\Sigma_{x}(g)$ the spectrum of $\mathcal{H}(g,x)$.

Instead of considering the operator on $\ell^{2}(\Z)$,  physicists and researchers in the random matrix community are more interested in the finite size version, known as the truncated Hatano-Nelson model, denoted by:
\begin{equation*}
[\mathcal{H}_{n}(g,x)u]_{k}=-\mathrm{e}^{g}u_{k+1}-\mathrm{e}^{-g}u_{k-1}+v(T^{k}x)u_{k}, \quad 0\leqslant k\leqslant n-1,
\end{equation*}
with the periodic boundary conditions\footnote{It is worth noting that if the boundary conditions \eqref{pbc} are replaced by Dirichlet boundary conditions, where $u_{-1}=u_{n}=0$, then the spectrum of the finite size Hatano-Nelson model is the same for all $g$. Further discussions on this matter can be found in Appendix \ref{whs}. }
\begin{equation}\label{pbc}
u_{0}=u_{n} \quad \text{and} \quad u_{-1}=u_{n-1},
\end{equation}
which can be represented as by the matrix
\begin{equation*}
\mathcal{H}_{n}(g,x)=\begin{pmatrix}
v(T^{0}x)&-\mathrm{e}^{g}&0&\cdots&-\mathrm{e}^{-g}\\
-\mathrm{e}^{-g}&v(T^{2}x)&-\mathrm{e}^{g}&\cdots&0\\
\vdots&\ddots&\ddots&\ddots&\vdots\\
0&\cdots&-\mathrm{e}^{-g}&v(T^{n-2}x)&-\mathrm{e}^{g}\\
-\mathrm{e}^{g}&\cdots&0&-\mathrm{e}^{-g}&v(T^{n-1}x)
\end{pmatrix}.
\end{equation*}

Let $\Sigma_{n,x}(g):=\{ E_{j}(g,x)\}_{1\leqslant j\leqslant n}$ be the spectrum (eigenvalues) of $\mathcal{H}_{n}(g,x)$. The central issue is the limit of the spectrum of  $\mathcal{H}_{n}(g,x)$ as $n\rightarrow \infty$. More precisely people are interested in information on the asymptotic distribution of the eigenvalues of $\mathcal{H}_{n}(g,x)$ as $n\rightarrow \infty$.

We emphasize that the limit of the spectrum of  $\mathcal{H}_{n}(g,x) $ may not coincide with   the spectrum of $\mathcal{H}(g,x)$. An interesting question is when the spectrum of the Hatano-Nelson model in $\ell^{2}(\mathbb{Z})$ can be approximated by the spectrum of its finite-interval truncation with periodic boundary conditions.

All these problems have been studied for the Hatano-Nelson model with random potentials \cite{GK00,GK03,GS18}. In this paper, we investigate these problems for the Hatano-Nelson model with strictly ergodic potentials, which include quasi-periodic potentials as a special case. We will see that the spectral phenomenon of random potentials and strictly ergodic potentials are strikingly different.

\subsection{The Hatano-Nelson model on $\ell^{2}(\mathbb{Z})$ with strictly ergodic potentials}

Let $X$ be a topological space and $T: X\rightarrow X$.  $(X, T)$  is said to be \textit{strictly ergodic} if  $(X, T)$ is both minimal and uniquely ergodic. Typical examples are quasi-periodic translations on the torus and standard skew-shift. We say $V(\cdot)$ is strictly ergodic if $V(n)=v(T^{n}x)$ with $v\in C^{0}(X, \mathbb{C})$. It is easy to check that random potentials are dynamical defined, but not strictly ergodic. It is worth noting that if $T$ is minimal, then there exists a subset $\Sigma(g)\subseteq\mathbb{C}$ such that $\Sigma_{x}(g)=\Sigma(g)$ for all $x\in X$ \cite{BLLS}.

To introduce our findings, the central concept we will be discussing is the Lyapunov exponent. The eigenfunction equation   $\mathcal{H}(g,x)u=Eu$ can be expressed as:
\begin{equation*}
\begin{pmatrix}
u_{n}\\
u_{n-1}
\end{pmatrix}=A_{n}^{g}(x)
\begin{pmatrix}
u_{0}\\
u_{-1}
\end{pmatrix}=\prod_{k=n-1}^{0} \begin{pmatrix}
\mathrm{e}^{-g}(v(T^{k}x)-E)&-\mathrm{e}^{-g}\\1&0
\end{pmatrix}
\begin{pmatrix}
u_{0}\\
u_{-1}
\end{pmatrix}.
\end{equation*}
The (maximum) Lyapunov exponent at $E$ is defined as:
\begin{equation*}
L_{g}(E)=\lim_{n\rightarrow \infty} \int_{X}  \frac{1}{n}\log\|A^{g}_{n}(x)\|\mathrm{d}\mu(x).
\end{equation*}
In the case when $g=0$, we abbreviate $A^{0}_{n}$ as $A_{n}$ and  $L_{0}(E)$ as $L(E)$. We can separate $\mathbb{C}$ into three sets: $\mathbb{C}= \mathcal{E}_{-} \cup \mathcal{E}_{0}\cup \mathcal{E}_{+}$, where
\begin{equation}\label{decomposition}
\begin{split}
&\mathcal{E}_{-}=\{E: L(E)< g\}, \ \mathcal{E}_{0}=\{E: L(E)= g\},\ \mathcal{E}_{+}=\{E: L(E)> g\}.
\end{split}
\end{equation}

Now we can present our main result as follows:

\begin{theorem}\label{spec}
Let $(X,T)$ be strictly ergodic and $X$ be a compact and connected metric space. Let $v\in C^{0}(X,\mathbb{C})$ and $g>0$. Then
\begin{equation*}
\Sigma(g)=\mathcal{E}_{0} \cup (\Sigma(0)\cap \mathcal{E}_{+}).
\end{equation*}
\end{theorem}

Let us first compare the spectral phenomenon discovered in this paper with Hatano-Nelson model with random potentials. It is well-known that for random potentials, there exists a compact subset $\Sigma(g)\subseteq \mathbb{C}$ such that $\Sigma_{x}(g) = \Sigma(g)$ with probability one. Davies \cite{Dav01A} proved that if $V$ is i.i.d. according to a probability measure $\mu$ with a compact support $M=[-\lambda,\lambda]$, then with probability one,
\begin{equation}\label{tuo}
\Sigma(g)\supseteq O+M,
\end{equation} 
where $O$ is the ellipse $\{\mathrm{e}^{g+\mathrm{i}\theta}+\mathrm{e}^{-g-\mathrm{i}\theta}: \theta\in[0,2\pi]\}$.
Moreover, if $\lambda\geqslant \mathrm{e}^{g}+\mathrm{e}^{-g}$, then
\begin{equation}\label{tuo2}
\Sigma(g)=O+M.
\end{equation}
It is obvious that the spectrum formula of $\Sigma(g)$ in Theorem \ref{spec} is different from that in \eqref{tuo} or \eqref{tuo2} with i.i.d. potentials.  Moreover, from \eqref{tuo2} it is easy to see that for real-valued i.i.d. potentials, $\Sigma(g)$ has a complex spectrum even $m(\Sigma(g))>0$ for any $g>0$, where $m$ represents the two-dimensional Lebesgue measure in the complex plane $\mathbb{C}$. However, in the following, we will see that for real-valued strictly ergodic potentials, $\Sigma(g)$ has a real-complex spectrum transition when $g$ varies.

If $v$ is real-valued, following the custom in \cite{GK00,GK03,GS18}, $\Sigma(0)\cap \mathcal{E}_{+}$ is called the real component and $\mathcal{E}_{0}$ is called the complex component (even if $\mathcal{E}_{0}$ can intersect $\mathbb{R}$). 
Whether  $\mathcal{H}(g,x)$ has  real-complex spectrum transition is significant in non-Hermitian physics as it indicates a change in the topology of the spectrum \cite{Gong, longhi}. For {\it real-valued} strictly ergodic potentials, Theorem \ref{spec} implies the real-complex spectrum transition when $g$ varies:

\begin{theorem}\label{transition}
Let $(X,T)$ be strictly ergodic and $X$ be a compact and connected metric space. Let $v\in C^{0}(X,\mathbb{R})$ and $g> 0$. Denote 
\begin{equation}\label{gcr}
\underline{g}_{cr}= \inf\{L(E): E\in \mathbb{R}\}\footnote{In general, we have $\underline{g}_{cr} \leqslant \inf\{L(E): E\in\Sigma(0)\}$ and the equality holds if $L(E)$ is continuous.}, \ \
\overline{g}^{cr}= \sup\{L(E): E\in \Sigma(0)\}.
\end{equation} 
Then we have the following:
\begin{enumerate}[font=\normalfont, label={(\arabic*)}]
\item \label{item:real}If $g\leqslant \underline{g}_{cr}$, then $\Sigma(g)=\Sigma(0)\subseteq\mathbb{R}$.


\item \label{item:real-comp}If $\underline{g}_{cr}<g<\overline{g}^{cr}$, then $\Sigma(g)$ contains both complex and real components.

\item \label{item:comp}If $g\geqslant \overline{g}^{cr}$, then $\Sigma(g)=\mathcal{E}_{0}$ only contains  complex components.
\end{enumerate}
\end{theorem}

\begin{remark}
To the best knowledge of the authors,  Theorem \ref{transition} presents the first result of the real-complex spectrum transition for the Hatano-Nelson model on $\ell^{2}(\mathbb{Z})$.
\end{remark}

Once the operator $\mathcal{H}(g,x)$ has a complex spectrum, it is interesting to study the structure of the complex spectrum $\mathcal{E}_{0}$. For this purpose, let us recall the  well-known Thouless formula. Indeed,
if the potential is  real-valued, the Thouless formula says that 
\begin{equation*}
L(E)=\int_{-\infty}^{\infty}\log |z-E|\mathrm{d}N(z).
\end{equation*}
where $N(E)$ is the   integrated density of states  for the self-adjoint Schr\"odinger operator $H(x):=\mathcal{H}(0,x)$. By the upper semi-continuity of $L(E)$, $\mathbb{R}\cap \mathcal{E}_{-}$ consists of disjoint open intervals, i.e., $\mathbb{R}\cap \mathcal{E}_{-}=\cup_{j}(a_{j},b_{j})$. Moreover,  as a direct consequence of the Thouless formula, we have the following lemma.
\begin{lemma}[\cite{GK03}, Proposition 3.1]\label{analytic}
The following conclusions hold:
\begin{enumerate}[font=\normalfont, label={(\arabic*)}]
\item \label{item:mono} $L(E+\mathrm{i}\eta)=L(E-\mathrm{i}\eta)$, ${\partial_\eta} L(E+\mathrm{i}\eta)>0$, and $L(E+\mathrm{i}\eta)>L(E)$ for any $E\in\mathbb{R}$ and $\eta>0$.
\item \label{item:curve} The level set $$\{E\in \mathbb{C}: L(E)= g\} \cap \{E:\mathrm{Im}E>0\}$$ consists of disjoint analytic arcs, denoted by
\begin{equation*}
\mathrm{Im}E=f_{j} (\mathrm{Re} E),\ \ a_{j}<\mathrm{Re}E<b_{j},
\end{equation*}
whose end points lie on $\mathbb{R}$, i.e., $f_{j}(a_{j}+0)=f_{j}(b_{j}-0)=0$, if $|a_{j}|, |b_{j}|<\infty$.
\end{enumerate}	
\end{lemma}

Lemma \ref{analytic} implies that $\mathcal{E}_{0}$ consists of a family of contours for real-valued strictly ergodic  potentials, and consequently $m(\Sigma(g))=0$. Therefore, Theorem \ref{transition} reveals that the spectral phenomenon of random potentials and strictly ergodic potentials are strikingly different from the random case where it always holds that $m(\Sigma(g))>0$.  However, as we will explain in section \ref{finite},  in the random setting, this kind of real-complex phase transition holds for $\mathcal{H}_n(g,x)$ in the large $n$ limit \cite{GK00, GK03}.

On the other hand, if $L(E)$ is continuous with respect to  $E$, then these contours are closed.  Thus, the natural question is whether the number of closed contours is finite. Indeed, if the potential is random, as argued by Goldsheid and Khoruzhenko \cite{GK00},  ``It is easy to construct examples with a prescribed finite number of contours. However we do not know any obvious reason for the number of contours to be finite for an arbitrary distribution''.
In this paper, we will provide a partial answer to Goldsheid and Khoruzhenko's question \cite{GK00} in the strictly ergodic  setting. 

\begin{corollary}\label{contour}
Under the assumption of Theorem \ref{transition}, assume further that $L(E)$ is continuous w.r.t. $E$. If $g>\overline{g}^{cr}$, then $\mathcal{E}_{0}$ consists of a finite number of contours. 
\end{corollary}

\begin{remark}
We should point out that the continuity of the Lyapunov exponent is a very important topic in dynamical systems. In the random setting, it is always continuous \cite{AEV,bv}. However, if the base dynamics is quasi-periodic, it depends sensitively on the regularity of the fiber \cite{BJ02, WY}.
\end{remark}

Theorem \ref{transition} not only reveals the real-complex spectrum transition, but also provides a mechanism for generating complex spectrum (the appearance of $\mathcal{E}_{0}$), even two-dimensional complex spectrum  ($m(\Sigma(g))>0$). We illustrate this with the example of quasi-periodic potentials.

\begin{corollary}\label{single}
	Let $\alpha\in\mathbb{R}\backslash\mathbb{Q}$ and $v(x)=\lambda \mathrm{e}^{2\pi \mathrm{i}x}$ with $|\lambda|>1$. Then $m(\Sigma(g))>0$ if and only if $g=\log|\lambda|$, where 
	 $\displaystyle\mathcal{E}_{0}=\bigg\{E\in\mathbb{C}: \Big(\frac{\mathrm{Re}E}{\cosh g}\Big)^{2}+\Big(\frac{\mathrm{Im}E}{\sinh g}\Big)^{2}\leqslant 4\bigg\}.$
\end{corollary}
\begin{remark}
Corollary \ref{single} was first proved in \cite{ZFW},  here we give a new proof. 
\end{remark}

For a more general case, we present another interesting example using the quantitative global theory of one-frequency Schr\"odinger operators \cite{GJYZ}. Consider the one-frequency Hatano-Nelson model with a complex extension  
\begin{equation}\label{extenHN}
[\mathcal{H}(g,x+\mathrm{i}y)u]_{n}=-\mathrm{e}^{g}u_{n+1}-\mathrm{e}^{-g}u_{n-1}+v(x+ n\alpha+\mathrm{i}y)u_{n}.
\end{equation}  
If $v$ is a trigonometric polynomial, the dual operator of $H(x)$ can be expressed as
\begin{equation*}
[\widehat{L}(x)u]_{n}=-\sum_{k=-m}^{m}  \widehat{v}_{k}u_{n+k} + 2\cos2\pi(x+ n\alpha)u_{n}.
\end{equation*}
Recalling that the eigenfunction equation $\widehat{L}(x)u=Eu$ leads to a cocycle denoted by $(\alpha,\widehat{L}_{E,v}^{2\cos})$, we note that $(\alpha,\widehat{L}_{E,v}^{2\cos})$ is symplectic and its  Lyapunov exponents come into pairs. Without loss of generality, we assume the non-negative part $0\leqslant \gamma_{1}(E)\leqslant \cdots\leqslant \gamma_{m}(E)$ according to multiplicity. Using the quantitative global theory of one-frequency Schr\"odinger operators \cite{GJYZ}, we prove the following result:

\begin{corollary}\label{2d}
Let $(\alpha, v)\in \mathbb{R}\backslash\mathbb{Q} \times C^{\omega}(\mathbb{T},\mathbb{R})$ with 
\begin{equation*}
v(x)=\sum_{k=-m}^{m} \widehat{v}_{k} \mathrm{e}^{2\pi \mathrm{i} kx}.
\end{equation*}
Let $g=2\pi my+\log |\widehat{v}_{m}|$ in \eqref{extenHN}. Then  $m(\Sigma(g,y))>0$ for any $y>\gamma_{m}(0)$, where $\Sigma(g,y)$ is the spectrum of $\mathcal{H}(g,x+\mathrm{i}y)$.
\end{corollary}

\subsection{Finite size Hatano-Nelson model}\label{finite}

It is worth noting that for real-valued random potentials, \eqref{tuo} and \eqref{tuo2} imply that the spectrum of $\mathcal{H}(g)$ is always two-dimensional. Therefore, regarding the random matrix $\mathcal{H}_{n}(g,x)$, a fundamental question arises: does the spectrum of $\mathcal{H}_{n}(g,x)$ converge to the spectrum of $\mathcal{H}(g)$ as $n\rightarrow \infty$? 

To address this question, Goldsheid and Khoruzhenko \cite{GK00, GK03} introduced the concept of \textit{selfaveraging} potentials. A real-valued potential is selfaveraging if the integrated density of states $N(E)$ of the Schr\"odinger operator with this potential exists. The class of selfaveraging potentials is vast, encompassing decaying potentials, periodic and almost-periodic potentials, as well as stationary random potentials (see \cite{PF} for proofs and additional examples). For selfaveraging potentials, if $g <\underline{g}_{cr}$, the eigenvalues of $\mathcal{H}_n(g,x)$ are asymptotically real, as shown in \cite{GK03, GK00, GS18}. If $\underline{g}_{cr} < g < \overline{g}^{cr}$, a finite proportion of eigenvalues are asymptotical to non-random analytic curves in the complex plane. Goldsheid and Khoruzhenko \cite{GK00, GK03} demonstrated that these curves converge to non-random limiting curves as $n\rightarrow \infty$. If $g > \overline{g}^{cr}$, nearly all eigenvalues shift away from the real axis. Here, $\underline{g}_{cr}$ and $\overline{g}^{cr}$ are defined in \eqref{gcr}. In particular, the results in \cite{GK03, GK00, GS18} imply that $\Sigma(g)$ cannot be approximated by $\Sigma_{n,x}(g)$.

In comparison to Theorem \ref{transition}, \cite{GK03, GK00} provides the real-complex spectrum transition for the limit of the spectrum of $\mathcal{H}_{n}(g,x)$ as $n\rightarrow\infty$ for selfaveraging potentials, while Theorem \ref{transition} gives a real-complex spectrum transition of the Hatano-Nelson model on $\ell^2(\Z)$ for strictly ergodic potentials. According to Davies's result \cite{Dav01A}, the spectrum of $\mathcal{H}_{n}(g,x)$ does not converges to the spectrum of $\mathcal{H}(g)$ for random potentials. A natural question is whether the spectrum of $\mathcal{H}_{n}(g,x)$ converge to the spectrum of $\mathcal{H}(g)$ for strictly ergodic potentials. We will give an affirmative answer in this paper.

Before that, let us introduce the concept of the density of states measure for $\mathcal{H}(g,x)$, which is a crucial quantity that characterizes the distribution of the spectrum.  Let $\mathrm{d}\mathcal{N}^{g,x}_{n}$ represent the normalized eigenvalue counting measure:
\begin{equation*}
	\mathrm{d}\mathcal{N}^{g,x}_{n} = \frac{1}{n}\sum_{j=1}^{n}\delta_{E_{j}(g,x)}.
\end{equation*}
As proved by Goldsheid and Khoruzhenko \cite{GK00}, for ergodic potential,  $\mathrm{d}\mathcal{N}^{g,x}_{n}$ weakly converges to the measure $\mathrm{d}\mathcal{N}^{g}$, which we call the  density of states  of $\mathcal{H}(g,x)$. Avron and Simon \cite{AS83} proved that for self-adjoint ergodic Schr\"odinger operators, the almost surely spectrum is the support of the density of states measure, see also \cite{DFbook}, thereby guaranteeing the convergence of the spectrum as $n\rightarrow\infty$. This is not the case for the Hatano-Nelson model with random potential. However, we prove it is also the case for the Hatano-Nelson model with strictly ergodic potential, which implies that the limit of the spectrum of  $\mathcal{H}_n(g,x)$ coincides with the spectrum of $\mathcal{H}(g,x)$.

\begin{theorem}\label{asy}
Under the assumption of Theorem \ref{transition}, assume further that $L(E)$ is continuous w.r.t. $E$. Then  
\begin{equation*}
\supp \mathrm{d}\mathcal{N}^{g} =\Sigma(g).
\end{equation*}	
\end{theorem}

For a real-valued i.i.d. random potential, it is expected that $\supp\mathrm{d}\mathcal{N}^{g}=\mathcal{E}_{0}\cup (\Sigma(0)\cap\mathcal{E}_{+})$. Goldsheid and Khoruzhenko \cite{GK00,GK03} proved that, with probability one, the following holds: for every bounded continuous function $f$,
\begin{equation*}
\lim_{n\rightarrow \infty} \int_{\mathbb{C}}f(E) \mathrm{d} \mathcal{N}^{g,x}_{n}=\int_{\Sigma(0)\cap\mathcal{E}_{+}} f(E) \mathrm{d}N+ \int_{\mathcal{E}_{0}} f(E(l)) \rho(E(l)) \mathrm{d}l,
\end{equation*}
where $\mathrm{d}l$ is the arc-length element on $\mathcal{E}_{0}$ and 
$
\rho(E)=\frac{1}{2\pi}  \big| \int_{\mathbb{R}} \frac{\mathrm{d}N(\lambda)}{\lambda-E}\big| \ \text{for}\ E\notin\mathbb{R}.
$ Recently, Goldsheid and Sodin \cite{GS18} showed that  if $j$ is fixed and $g$ varies from $0$ to $\infty$, the eigenvalue $E_{j}(g,x)$ remains real and exponentially close to $E_{j}(0,x)$ for $g \leqslant  L(E_{j}(0,x)) - \epsilon$.
Compare to \cite{GK00,GK03,GS18}, Theorem \ref{asy} points out the exact location of the limit of $\Sigma_{n,x}(g)$ for all parameter $g>0$.

\section{Preliminary}

\subsection{Cocycles and uniform hyperbolicity}
In the rest of the paper, we always assume that $X$ is a topological space and 
$(X,\mu,T)$ is ergodic. For $A\in C^{0}(X,\mathrm{GL}(2,\mathbb{C}))$, one can define the cocycle $(T, A)$ as follows:
\begin{eqnarray*}
(T, A):\left\{\begin{array}{c}
X \times \mathbb{C}^2 \rightarrow X \times \mathbb{C}^2, \\
(x, v) \mapsto(Tx, A(x) \cdot v).
\end{array}\right.	
\end{eqnarray*}	
The iterates of $(T, A)$ are of the form $(T, A)^n=(T^{n}, A_n)$, where
\begin{eqnarray*}
A_n(x):= \begin{cases}A(T^{n-1}x) \cdots A(Tx) A(x), & n > 0,\\
I_{2},	&n=0, \\
A(T^{-n}x)^{-1} \cdots A(T^{-1}x)^{-1} , & n<0.\end{cases}
\end{eqnarray*}
The (maximum) Lyapunov exponent of $(T,A)$ is defined as
\begin{equation*}
L(T,A)=\lim_{n\rightarrow \infty}\int_{X} \frac{1}{n}\log \|A_{n}(x)\| \mathrm{d}\mu(x).
\end{equation*}
The existence of the above limit is ensured by the Kingman subadditive ergodic theorem \cite{King73}. For $X=\mathbb{T}^{d}$ and $T:x\rightarrow x+\alpha$ with $(1,\alpha)$ rationally independent, we also denote $L(\alpha, A):=L(T,A)$.

Recall that $(T, A)$ is uniformly hyperbolic if there exist $u,s\in C^{0}(X,\mathbb{PC}^{2})$, called the unstable and stable directions, and $n \geqslant 1$ such that
\begin{enumerate}[font=\normalfont, label={(\arabic*)}]
\item for every $x \in X$ we have	$A(x)\cdot u(x)=u(Tx)$ and $A(x) \cdot s(x)=s(Tx),$
\item for every unit vector $w \in s(x)$ we have $\|A_n(x)  w\|<1$,
\item for every unit vector $w \in u(x)$ we have $\|A_n(x) w\|>1$.
\end{enumerate}
From now on, $(T, A)\in \mathcal{UH}$ means that $(T,A)$ is uniformly hyperbolic.

\begin{lemma}[\cite{av11}, Lemma 2.3]\label{UHharmonic}
Let $X$ be a compact metric space and $T: X\rightarrow X$ be a homeomorphism, and let $\mu$ be a probability measure invariant by $T$. The Lyapunov exponent $A\mapsto L(T,A)$ is a plurisubharmonic function of $A\in C^{0}(X,\mathrm{SL}(2,\mathbb{C}))$. Moreover, restricted to the set of uniformly hyperbolic $A$, it is a pluriharmonic function. 
\end{lemma}
\subsection{Dynamical defined operators}

We consider the complex-valued dynamical defined Hatano-Nelson model
\begin{equation*}
[\mathcal{H}(g,x)u]_{n}=-\mathrm{e}^{g}u_{n+1}-\mathrm{e}^{-g}u_{n-1}+v(T^{n}x) u_{n}.
\end{equation*}
If $g=0$, it reduces to the complex-valued dynamical defined Schr\"odinger operator:
\begin{equation*}
[H(x)\psi]_n=-\psi_{n+1}-\psi_{n-1}+v(T^nx),\quad n\in \mathbb{Z}.
\end{equation*}
We have the following relation regarding the solutions of eigenfunction equations of the two operators.
\begin{lemma}\label{subs}
Let $x\in X$, $g> 0$, and $E\in\mathbb{C}$. Then $u$ satisfies $\mathcal{H}(g,x)u=Eu$ if and only if $H(x)\psi=E\psi$ where $\psi=(\psi_{n})$ and  $\psi_{n}=\mathrm{e}^{ng}u_{n}$.
\end{lemma}

Since $H(x)u=Eu$  naturally defines a discrete dynamical system
\begin{equation*}
\begin{pmatrix}
\psi_{n+1}\\
\psi_{n}
\end{pmatrix}=\begin{pmatrix}
v(T^n x)-E&-1\\
1&0
\end{pmatrix}\begin{pmatrix}
\psi_{n}\\
\psi_{n-1}
\end{pmatrix},
\end{equation*}
which derives an $\mathrm{SL}(2,\mathbb{C})$-valued cocycle $(T,S_{E})$ with $
S_{E}(x)=\begin{pmatrix}
v(x)-E & -1 \\
1 & 0
\end{pmatrix}$.
We call $(T, S_{E})$ the Schr\"odinger cocycle.  Similarly, one can also define the  Hatano-Nelson cocycle $(T,S_{E}^{g})$ where $
S_{E}^{g}(x)= \begin{pmatrix}
\mathrm{e}^{-g}[v(x)-E]&-\mathrm{e}^{-2g}\\
1&0
\end{pmatrix}$. Moreover, if $T: x\rightarrow x+\alpha$, then we denote $(T, S_{E})$ (resp. $(T, S_{E}^{g})$) as $(\alpha, S_{E})$ (resp. $(\alpha, S_{E}^{g})$).

Denote by $\Sigma_{x}(g)$  the spectrum of $\mathcal{H}(g,x)$. If $T$ is minimal, then there exists a subset $\Sigma(g)\subseteq\mathbb{C}$ such that $\Sigma_{x}(g)=\Sigma(g)$ for all $x\in X$ \cite{BLLS}. In the self-adjoint case, i.e. the potential $v$ is real-valued and $g=0$, then by the well-known result of Johnson \cite{John86},  $E\notin \Sigma (0)$ if and only if $(T,A)\in\mathcal{UH}$. This result can be generalized  to the non-self-adjoint case:
\begin{proposition}[\cite{GJYZ}, Theorem 2.4]\label{equi} 
Let $(X,T)$ be minimal and $v\in C^{0}(X,\mathbb{C})$. Then $E\notin \Sigma(0)$ if and only if $(T,A)\in \mathcal{UH}$.
\end{proposition}

\subsection{Subharmonic functions}

The following basic properties of the subharmonic functions in potential theory are well-known.
\begin{lemma}[\cite{Ransford}, Theorem 3.1.2, Corollary 3.7.5]\label{potential}
For the finite Borel measure $\mu$ on $\mathbb{C}$ with compact support, the  potential of $\mu$  defined by
\begin{equation*}
p_{\mu}(z) = \int_{\mathbb{C}} \log |\zeta-z| \mathrm{d}\mu(\zeta)
\end{equation*}
has the following properties:
\begin{enumerate}[font=\normalfont, label={(\arabic*)}]
\item \label{item:subharm}The potential $p_{\mu}$ is subharmonic on $\mathbb{C}$, and harmonic on $\mathbb{C}\backslash \supp\mu$.
\item \label{item:equ}Let $\mu_1$ and $\mu_2$ be finite Borel measures on $\mathbb{C}$ with compact support. If $p_{\mu_1}=p_{\mu_2}+h$ on an open set $U$, where $h$ is harmonic on $U$, then $\mu_1|_U=\mu_2|_U.$

\item \label{item:growth}$\displaystyle p_{\mu}(z)=\mu(\mathbb{C})\log|z|+O(|z|^{-1})$ as $|z|\rightarrow \infty$.
\end{enumerate}
\end{lemma}

\begin{lemma}[\label{schwartz}Lower envelope theorem, \cite{Lan}, Theorem 3.8]
	Suppose that $\{\mathrm{d}\mu_{n}\}_{n\geqslant 1}$ has compact support in $\mathbb{C}$. Denote $\displaystyle p_{n}(z)=\int_{\mathbb{C}} \log |z-\zeta|\mathrm{d}\mu_{n}(\zeta)$.
	Assume that $\mathrm{d}\mu_{n} \rightharpoonup \mathrm{d}\mu$, as $n\rightarrow\infty$. Then
	\begin{equation*}
		\limsup_{n\rightarrow\infty} p_{n}(z)= \int_{\mathbb{C}} \log |z-\zeta|\mathrm{d}\mu(\zeta)\quad\text{ for $m$-a.e. } z\in\mathbb{C}.
	\end{equation*}
\end{lemma}

\section{ Hatano-Nelsen model on $\ell^{2}(\mathbb{Z})$}

In this section, we analyze the spectrum set of the  dynamical defined Hatano-Nelson model \eqref{ddhn}, and prove Theorem \ref{spec}. Before giving its proof, we first give a useful comment, as it will be heuristic for the proof. Theorem \ref{spec} says that 
\begin{equation*}
\Sigma(g) = \mathcal{E}_{0} \cup (\Sigma(0)\cap\mathcal{E}_{+}).
\end{equation*}
If $g=0$, $\mathcal{E}_{0}$ corresponds to the Schr\"odinger cocycle $(T, S_{E})$ which has zero Lyapunov exponent, and $\Sigma(0)\cap\mathcal{E}_{+}$ corresponds to the Schr\"odinger  cocycle $(T, S_{E})$  which is non-uniformly hyperbolic. Once we have this, 
Theorem \ref{spec}   reduces  to the  following two propositions.
\begin{proposition}\label{supset}
Let $(X,T)$ be strictly ergodic and $X$ be a compact and connected metric space. Let $v\in C^{0}(X,\mathbb{C})$ and $g> 0$. Then
\begin{equation*}
\mathcal{E}_{0} \cup (\Sigma(0) \cap \mathcal{E}_{+})\subseteq\Sigma(g).
\end{equation*}
\end{proposition}

\begin{proposition}\label{subset}
Let $(X,T)$ be strictly ergodic and $X$ be a compact and connected metric space. Let $v\in C^{0}(X,\mathbb{C})$ and $g> 0$. Then
\begin{equation*}
\Sigma(g) \subseteq \mathcal{E}_{0} \cup (\Sigma(0)\cap\mathcal{E}_{+}).
\end{equation*}
\end{proposition}

\subsection{Proof of Proposition \ref{supset}}
\begin{proof}

Let us recall the concept of exponential dichotomy.
We say that $(T, A)$ has exponential dichotomy if there are positive constants $C>0, \gamma>0$ and a continuous, projection-valued function $x \mapsto P_x=P_x^2: X \rightarrow \mathbb{PC}^{2}$ such that
\begin{equation*}
\begin{array}{cc}
\|A_n(x) P_x A_m(x)^{-1}\| \leqslant C \mathrm{e}^{-\gamma(n-m)}, & n \geqslant m, \\
\|A_n(x)(\mathrm{Id}-P_x) A_m(x)^{-1}\| \leqslant C \mathrm{e}^{\gamma(n-m)}, & n \leqslant m,
\end{array}
\end{equation*}
for all $x \in X$ and $m, n \in \mathbb{Z}$. For $\mathrm{SL}(2,\mathbb{C})$-valued cocycles, exponential dichotomy is also called uniform hyperbolicity, see \cite{John82}.

Recall that the Sacker-Sell spectrum of the cocycle $(T,A)$ is defined as
\begin{equation*}
\Sigma_{SS}(A)=\{\lambda\in \mathbb{R}: (T, \mathrm{e}^{\lambda}A) \  \text{has no exponential dichotomy}\}.
\end{equation*}
In particular, the Sacker-Sell spectrum of any $\mathrm{GL}(2,\mathbb{C})$-valued cocycle is the union of at most two intervals. The following lemma relates the Lyapunov exponent and the Sacker-Sell spectrum.
\begin{lemma}[\label{SSS}\cite{JPS87}, Theorem 2.3]
Let $(X,\mu,T)$ be ergodic with $X$ compact and connected and $A\in C^{0}(X,\mathrm{SL}(2,\mathbb{C}))$. Then
\begin{equation*}
\partial \Sigma_{SS} \subseteq \Sigma_{L} \subseteq \Sigma_{SS},
\end{equation*}
where $\Sigma_{L}(A)=\{\pm L(T,A)\}$ is called the Lyapunov spectrum of $(T,A)$.
\end{lemma}
Fix $E\in\mathcal{E}_{0}\cup(\Sigma(0)\cap\mathcal{E}_{+})$. Obviously, the Lyapunov spectrum of $(T,S_{E})$ is $\Sigma_{L}(S_{E})=\{\pm L(E)\}$ since $\det S_{E}=1$. Thus by Lemma \ref{SSS}, if $E\in\mathcal{E}_{0}$, then
\begin{equation*}
-g=-L(E)\in\Sigma_{SS}(S_{E}).
\end{equation*}
If $E\in \Sigma(0)\cap\mathcal{E}_{+}$, then $0\in \Sigma_{SS}(S_{E})$ and we deduce that 
\begin{equation*}
-g\in [-L(E),L(E)]=\Sigma_{SS}(S_{E}).
\end{equation*}
It is obvious that $-g\in \Sigma_{SS}(S_{E})$ in either case, and thus $(T, \mathrm{e}^{-g}S_{E})$ has no exponential dichotomy. Let $D=\begin{pmatrix}
\mathrm{e}^{-g/2}&0\\
0&\mathrm{e}^{g/2}
\end{pmatrix}$, then the direct calculation shows
\begin{equation*}
D^{-1} S_{E}^{g}(x)D=\mathrm{e}^{-g} S_{E}(x).
\end{equation*}
Since the exponential dichotomy of the cocycle is invariant by conjugation, which means the Hatano-Nelson cocycle $(T,S_{E}^{g})$ has no exponential dichotomy as well. 

To finish the proof, we will use the Weyl sequence argument. Before that, let us recall the following result of Sacker and Sell \cite{SS76}.
\begin{lemma}[\label{SSlemma}\cite{SS76}, Theorem 2]
If there is no non-trivial bounded solution $u$ satisfying $\mathcal{H}(g,x)u=Eu$ for some $x$, then $(T, S_{E}^{g})$ has exponential dichotomy.
\end{lemma}

According to Lemma \ref{SSlemma}, there exists $x_0\in X$ and a non-trivial bounded solution $u$ such that  $\mathcal{H}(g,x_0)u=Eu$. For any $i\geqslant1$, we let $\phi^{(i)}= \chi_{[-i,i]} u$.
Observe that
\begin{eqnarray*}
&& [(\mathcal{H}(g,x_{0})-E)\phi^{(i)}]_{i} = \mathrm{e}^{g}u_{i+1},  \qquad \quad
[(\mathcal{H}(g,x_{0})-E)\phi^{(i)}]_{i+1}=-\mathrm{e}^{-g}u_{i},\\
&& [(\mathcal{H}(g,x_{0})-E)\phi^{(i)}]_{-i} = \mathrm{e}^{-g}u_{-i-1}, \quad
[(\mathcal{H}(g,x_{0})-E)\phi^{(i)}]_{-i-1}=-\mathrm{e}^{g}u_{-i},
\end{eqnarray*}
which implies that
\begin{equation*}
\bigg\| (\mathcal{H}(g,x_{0})-E)\frac{\phi^{(i)}}{\|\phi^{(i)}\|_{2}}\bigg\|_{2}^{2} \leqslant \mathrm{e}^{g}\frac{\|\phi^{(i+1)}\|_{2}^{2}-\|\phi^{(i-1)}\|_{2}^{2}}{\|\phi^{(i)}\|_{2}^{2}} \leqslant \frac{C}{\|\phi^{(i)}\|_{2}^{2}}.
\end{equation*}
If $\lim_{i\rightarrow \infty}\|\phi^{(i)}\|_{2}= \infty$, then $\frac{\phi^{(i)}}{\|\phi^{(i)}\|_{2}}$ is the Weyl sequence and thus $E\in \Sigma_{x_{0}}(g)=\Sigma(g)$. If $\lim_{i\rightarrow \infty}\|\phi^{(i)}\|_{2}< \infty$, then $E$ is an eigenvalue, we also have $E\in \Sigma_{x_{0}}(g)=\Sigma(g)$.
\end{proof}

\subsection{Proof of Proposition \ref{subset}}
\begin{proof}
It suffices to prove that if 
$E\in\mathcal{E}_{-}\cup(\mathcal{E}_{+}\backslash \Sigma(0)),$ then $E\notin \Sigma(g)$. We distinguish the proof into two cases:\\

{\bf Case 1:} $E\in \mathcal{E}_{-}$.
Fix $x\in X$ and $n\in\mathbb{Z}$. We denote by $\phi^{n}_{(\cdot)}:\mathbb{Z}\rightarrow \mathbb{C}$ the solution of 
\begin{equation*}
[\mathcal{H}(g,x)\phi^{n}]_{(\cdot)}=E\phi^{n}_{(\cdot)}\ \text{with} \ \phi^{n}_{n}=0 \ \text{and}\ \phi^{n}_{n+1}=\mathrm{e}^{-g}.
\end{equation*}

Define $u^{n}_{m}=\mathrm{e}^{(m-n)g}\phi^{n}_{m}$ for any $m\in\mathbb{Z}$. Then by Lemma \ref{subs}, one can check that
\begin{equation*}
-u^{n}_{m+1}-u^{n}_{m-1}+v(T^{m}x)u^{n}_{m}=Eu^{n}_{m},\ \text{for any}\ m\in \mathbb{Z},
\end{equation*}
which means $u^{n}_{(\cdot)}$ is the solution of $[H(x)u^{n}]_{(\cdot)}=Eu^{n}_{(\cdot)}$ with $u^{n}_{n}=0$ and $u^{n}_{n+1}=1$. Hence
\begin{equation}\label{mul}
\begin{pmatrix}
\mathrm{e}^{(m-n+1)g}\phi^{n}_{m+1}\\
\mathrm{e}^{(m-n)g}\phi^{n}_{m}
\end{pmatrix}=
\begin{pmatrix}
u_{m+1}^{n}\\
u_{m}^{n}
\end{pmatrix}=
A_{m-n}(T^{n+1}x)
\begin{pmatrix}
1\\0
\end{pmatrix}.
\end{equation}

Now we define an infinite dimensional matrix $\mathcal{G}_{E,x}(\cdot,\cdot):\mathbb{Z}\times\mathbb{Z}\rightarrow\mathbb{C}$ as
\begin{equation}\label{Gx}
\mathcal{G}_{E,x}(m,n)=
\begin{cases}
0, &m\leqslant n,\\
-\phi^{n}_{m},&m\geqslant n+1.
\end{cases}
\end{equation}
The following lemma shows the boundedness of $\mathcal{G}_{E,x}$.
\begin{lemma}\label{bound1}
Let $g>0$, $E\in\mathcal{E}_{-}$ and $x\in X$. Then 
\begin{equation*}
\mathcal{G}_{E,x}: \ell^{2}(\mathbb{Z}) \rightarrow  \ell^{2}(\mathbb{Z}) \ \text{is bounded.}
\end{equation*}
\end{lemma}	
\begin{proof}
Let us show the exponential decay of the entries of $\mathcal{G}_{E,x}$. 
Recall that the Furman's Theorem \cite {Fur97} shows that for the Schr\"odinger cocycle $(T, S_{E})$,
\begin{equation*}
\lim_{N\rightarrow \pm\infty} \sup_{x\in X} \frac{1}{|N|} \log \|A_{N}(x)\| = L(E).
\end{equation*}
By the assumption that $E\in\mathcal{E}_{-}$, then for any $0<\varepsilon<g-L(E)$, there exists $\overline{N}_{1}=\overline{N}_{1}(E,v,g,\varepsilon)$ such that for any $m-n\geqslant \overline{N}_{1}$,
\begin{equation}\label{upep}
\sup_{x\in X} \frac{1}{m-n} \log \|A_{m-n}(x)\| \leqslant L(E)+\varepsilon.
\end{equation}
Therefore, combining \eqref{mul} with \eqref{upep}, for any $m-n>\overline{N}_{1}$, we have  
\begin{equation}\label{decay}
\begin{split}
\sup_{x\in X}|\mathcal{G}_{E,x}(m,n)|=\sup_{x\in X}|\phi^n_{m}|&\leqslant \mathrm{e}^{-(m-n)g}\sup_{x\in X}\|A_{m-n}(x)\| \leqslant \mathrm{e}^{-(g-L(E)-\varepsilon)(m-n)}.
\end{split}
\end{equation}
On the other hand, for any $1\leqslant m-n\leqslant \overline{N}_{1}$, we have
\begin{equation*}
\sup_{x\in X}|\mathcal{G}_{E,x}(m,n)|\leqslant \mathrm{e}^{-(m-n)g+C_{1}(m-n)}<\mathrm{e}^{C_{1} \overline{N}_{1}},
\end{equation*}
where $C_{1}:=\log (|E|+\|v\|_{C^{0}}+\mathrm{e}^{g}+\mathrm{e}^{-g})$. Therefore $\mathcal{G}_{E,x}$ defines an operator on $\ell^2(\mathbb{Z})$ in a natural way. According to Young's inequality, there exists $C_{2}=C_{2}(E,v,g)$ such that for any $\psi\in \ell^{2}(\mathbb{Z})$,
\begin{equation*}
\|\mathcal{G}_{E,x}\psi\|_{2}^{2}  =\sum_{m\in \mathbb{Z}}\Big|\sum_{n\in \mathbb{Z}} \mathcal{G}_{E,x}(m,n) \psi_{n} \Big|^{2}=\sum_{m\in \mathbb{Z}}\Big|\sum_{n\leqslant m-1} \mathcal{G}_{E,x}(m,n) \psi_{n} \Big|^{2}<C_{2}\|\psi\|_{2}^{2},
\end{equation*}
where the last inequality comes from  \eqref{decay}.
\end{proof}

Now, it follows from \eqref{Gx} that 
\begin{equation*}
[(\mathcal{H}(g,x)-E)\mathcal{G}_{E,x}](m,n) =\delta_{m}(n),
\end{equation*}
which means that $\mathcal{G}_{E,x}$ is the right inverse of $\mathcal{H}(g,x)-E$.  Thus, in order to prove $E\notin \Sigma(g)$, using Lemma \ref{bound1}, we only need to prove that $\mathcal{G}_{E,x}$ is also the left inverse of $\mathcal{H}(g,x)-E$.
\begin{lemma}\label{leftinv}
Let $g>0$, $E\in\mathcal{E}_{-}$ and $x\in X$. Then
\begin{equation*}
[\mathcal{G}_{E,x}(\mathcal{H}(g,x)-E)](m,n) =\delta_{n}(m).
\end{equation*}
\end{lemma}	
\begin{proof}
Obviously, 
\begin{equation*}
[\mathcal{G}_{E,x}(\mathcal{H}(g,x)-E)](m,n)=\delta_{mn}, \text{ for any } n\geqslant m.
\end{equation*}
When $m=n+1$ or $m=n+2$, by simple calculation, we have
\begin{equation*}
[\mathcal{G}_{E,x}(\mathcal{H}(g,x)-E)](m,n)=0.
\end{equation*}
When $m\geqslant n+3$, we recall that  $S_{E}^{g}(x)=\begin{pmatrix}
\mathrm{e}^{-g}[v(x)-E]&-\mathrm{e}^{-2g}\\
1&0
\end{pmatrix}$, then
\begin{equation*}
\begin{pmatrix}
\phi_{m}^{n-1}\\ \phi_{m-1}^{n-1}
\end{pmatrix} = \prod_{j=m-1}^{n} S_{E}^{g}(T^{j}x)\begin{pmatrix}
\phi_{n}^{n-1}\\ \phi_{n-1}^{n-1}
\end{pmatrix}  = \prod_{j=m-1}^{n} S_{E}^{g}(T^{j}x)\begin{pmatrix}
{\rm e}^{-g}\\0
\end{pmatrix} =: (I).
\end{equation*}
Similarly, we also get that
\begin{equation*}
\begin{split}
&\begin{pmatrix}
\phi_{m}^{n}\\ \phi_{m-1}^{n}
\end{pmatrix} = \prod_{j=m-1}^{n+1} S_{E}^{g}(T^{j}x)
\begin{pmatrix}
{\rm e}^{-g}\\0
\end{pmatrix}=:(II),\\
&\begin{pmatrix}
\phi_{m}^{n+1}\\ \phi_{m-1}^{n+1}
\end{pmatrix} = \prod_{j=m-1}^{n+2}S_{E}^{g}(T^{j}x)
\begin{pmatrix}
{\rm e}^{-g}\\0
\end{pmatrix}=:(III).
\end{split}
\end{equation*}
By  direct calculation, we have
\begin{equation}\label{expand}
-{\rm e}^{-g} (III) + [v(T^{n}x)-E] (II) - {\rm e}^{g} (I)=\bigg(\prod_{j=m-1}^{n+2}S_{E}^{g}(T^{j}x)\bigg) \Phi_{n}(x)
\begin{pmatrix}
{\rm e}^{-g}\\0
\end{pmatrix},
\end{equation}
where $\Phi_{n}(x)= -\mathrm{e}^{-g} I_{2} + [v(T^{n}x)-E]S_{E}^{g}(T^{n+1}x)- \mathrm{e}^{g}  S_{E}^{g}(T^{n+1}x) S_{E}^{g}(T^{n}x)$.  
Expanding $\Phi_{n}(x)$ shows that $\Phi_{n}(x)\begin{pmatrix}
{\rm e}^{-g}\\0
\end{pmatrix}=0$. 
Thus it follows from \eqref{expand} that
\begin{equation*}
-{\rm e}^{-g} \begin{pmatrix}
\phi_{m}^{n+1}\\ \phi_{m-1}^{n+1}
\end{pmatrix} +[v(T^{n}x)-E]\begin{pmatrix}
\phi_{m}^{n}\\ \phi_{m-1}^{n}
\end{pmatrix}-{\rm e}^{g} 
\begin{pmatrix}
\phi_{m}^{n-1}\\ \phi_{m-1}^{n-1}
\end{pmatrix}=0,
\end{equation*}
and the first row of the above equation gives 
\begin{equation*}
-{\rm e}^{-g}\phi_{m}^{n} +v(T^{n}x)\phi_{m}^{n} -{\rm e}^{g} \phi_{m}^{n-1}=E\phi_{m}^{n},
\end{equation*}
which shows the result.
\end{proof}

{\bf Case 2:} $E\in \mathcal{E}_{+}\backslash \Sigma(0)$. 
The following lemma shows $\mathcal{H}(g,x)-E$ has a bounded inverse and thus $E\notin \Sigma(g)$. 
\begin{lemma}\label{Gexists}
Let $g>0$, $E\in \mathcal{E}_{+}\backslash \Sigma(0)$ and $x\in X$. There exists a bounded linear operator $\mathcal{G}_{E,x}: \ell^{2}(\mathbb{Z})\rightarrow\ell^{2}(\mathbb{Z})$ such that 
\begin{equation}\label{iden}
[\mathcal{G}_{E,x}(\mathcal{H}(g,x)-E)](m,n) = [(\mathcal{H}(g,x)-E)\mathcal{G}_{E,x}](m,n) =\delta_{mn}.
\end{equation}
\end{lemma}
\begin{proof}
Proposition \ref{equi}  implies that $(T, S_{E})\in\mathcal{UH}$. Then by famous Oseledets-Ruelle theorem \cite{viana}, for any $x\in X$, there exist two unit vectors $\vec{s}(x):= \begin{pmatrix}
\varphi^{+}_{0}\\
\varphi^{+}_{-1}
\end{pmatrix}\in \mathbb{C}^{2}$ and $\vec{u}(x):=\begin{pmatrix}
\varphi^{-}_{0}\\
\varphi^{-}_{-1}
\end{pmatrix}\in \mathbb{C}^{2}$
such that
\begin{equation}\label{expdecay}
	\lim_{n\rightarrow\infty} \frac{1}{n}\log \|A_{n}(x)\vec{s}(x)\|_{2} =-L(E), \ \text{and}\ \lim_{n\rightarrow-\infty} \frac{1}{n}\log \|A_{n}(x)\vec{u}(x)\|_{2} =-L(E).
\end{equation}
Denote $\vec{\varphi}^{u}(n):=\begin{pmatrix}
	\varphi^{-}_{n}\\
	\varphi^{-}_{n-1}
\end{pmatrix}:= A_{n}(x)\vec{u}(x)$ and $\vec{\varphi}^{s}(n):=\begin{pmatrix}
	\varphi^{+}_{n}\\
	\varphi^{+}_{n-1}
\end{pmatrix}:= A_{n}(x)\vec{s}(x)$ for any $n\in \mathbb{Z}$. 

Now we construct $\mathcal{G}_{E,x}(\cdot,\cdot):\mathbb{Z}\times\mathbb{Z} \rightarrow \mathbb{C}$	by 
\begin{equation*}
	\mathcal{G}_{E,x}(m,n) :={\rm e}^{(n-m)g}G_{E,x}(m,n):=\frac{\mathrm{e}^{(n-m)g}}{\det (\vec{\varphi}^{u}(n),\vec{\varphi}^{s}(n) )}
	\begin{cases}
		\varphi^{-}_{m} \cdot \varphi^{+}_{n}, &m\leqslant n,\\
		\varphi^{-}_{n} \cdot \varphi^{+}_{m}, &n\leqslant m,
	\end{cases}
\end{equation*}
where $G_{E,x}$ is the Green's function of $H(x)$ at $E$. By using $G_{E,x}(H(x)-E)=(H(x)-E)G_{E,x}=\mathrm{Id}$, one can easily check that \eqref{iden} holds. Since $(T, S_{E})\in\mathcal{UH}$, we know $|\det (\vec{\varphi}^{u}(n),\vec{\varphi}^{s}(n) )|>\varkappa$ uniformly for some $\varkappa>0$.
Moreover, it follows from \eqref{expdecay} that for any $0<\varepsilon<L(E)-g$, there exists $C_{3}=C_{3}(E,v,g,T,\varepsilon,\varkappa)$ such that
\begin{equation*}
	|\mathcal{G}_{E,x}(m,n)| \leqslant \begin{cases}
		C_{3}{\rm e}^{-(L(E)-\varepsilon-g)(n-m)}, &  m\leqslant n,\\
		C_{3}{\rm e}^{-(L(E)-\varepsilon+g)(m-n)}, & n\leqslant m.
	\end{cases}	
\end{equation*}	
Again, by using Young's inequality, we prove the boundedness of $\mathcal{G}_{E,x}$.
\end{proof}	

Combining {\bf Case} 1 with {\bf Case} 2 proves Proposition \ref{subset}. 
\end{proof}

\section{Applications of Theorem \ref{spec}}

\subsection{Transition of real-complex spectrum}
Once we have Theorem \ref{spec}, we can prove the transition of real-complex component in $\Sigma(g)$ for the Hatano-Nelson model with real-valued potential.
Recall that the transition points are given by
\begin{equation*}
\underline{g}_{cr}= \inf\{L(E): E\in \mathbb{R}\}, \ \
\overline{g}^{cr}= \sup\{L(E): E\in \Sigma(0)\}.
\end{equation*}

\begin{proof}[Proof of Theorem \ref{transition}]
(1) For $0< g\leqslant \underline{g}_{cr}$, by Theorem \ref{spec}, it is obvious that
\begin{equation*}
\Sigma(0)\cap \mathcal{E}_{-}=\emptyset,\quad \Sigma(0)\cap\mathcal{E}_{0}\subseteq \Sigma(g),\quad \Sigma(0)\cap\mathcal{E}_{+}\subseteq \Sigma(g),
\end{equation*}
thus $\Sigma(0)\subseteq\Sigma(g)$. Note that $\mathcal{E}_{+}\cap\Sigma(0)\subseteq\Sigma(0)$, so we only need to prove $\mathcal{E}_{0}\subseteq\Sigma(0)$. If  $\mathcal{E}_{0}=\emptyset$, there is nothing to prove. We assume $\mathcal{E}_{0}\neq\emptyset$ and there exists $E'\notin \Sigma(0)$ such that $L(E')=g$,  then by Proposition \ref{equi}, $(T,S_{E'})\in\mathcal{UH}$. 

On the one hand, as $L(E')=g$,   by the definition of $\underline{g}_{cr}$ and Lemma \ref{analytic},  $L(\cdot)$ attains the global minimum at $E'$. On the other hand, by Lemma \ref{UHharmonic},  $L(T,A)$ is pluri-harmonic w.r.t. the second variable for $(T,A) \in\mathcal{UH}$, i.e. $L(E)$ is harmonic in a neighborhood of $E'$. As $L(\cdot)$ attains the minimum at $E'$,  by extremum principle of harmonic functions, $L(E)$ is locally constant in the neighborhood of $E'$, 
which contradicts to Lemma \ref{analytic}. Hence $\mathcal{E}_{0}\subseteq\Sigma(0)$. This proves Theorem \ref{transition}\ref{item:real}.

(2) For $\underline{g}_{cr}<g<\overline{g}^{cr}$, we claim that
$\mathcal{E}_{0}\neq \emptyset$. In fact, for any $0<\varepsilon < g-\underline{g}_{cr}$, there exists $E'\in \mathbb{R}$ such that $L(E')<\underline{g}_{cr}+\varepsilon<g$ by the definition of $\underline{g}_{cr}$. It is well-known that the Lyapunov exponent is continuous w.r.t. uniformly hyperbolic cocycles, so it follows that $L(E)$ is continuous w.r.t. $E\in\mathbb{C}\backslash \mathbb{R}$. By Lemma \ref{analytic} and the upper semi-continuity of $L(E)$, there exists $\eta'\in\mathbb{R}\backslash\{0\}$ such that $
L(E'+\mathrm{i}\eta')=g.$ 
Hence $E'+\mathrm{i}\eta'\in\mathcal{E}_{0}\backslash\mathbb{R}$, which means  $\Sigma(g)$ contains complex  spectrum.

On the other hand, by the upper semi-continuity of the Lyapunov exponent, there exists $E'\in\Sigma(0)$ such that
\begin{equation*}
L(E') = \sup\{L(E): E\in \Sigma(0)\}=\overline{g}^{cr} >g,
\end{equation*}
thus $\Sigma(0)\cap \mathcal{E}_{+}\neq\emptyset$, then $\Sigma(g)$ contains real spectrum. This proves Theorem \ref{transition}\ref{item:real-comp}.

(3) For $g\geqslant\overline{g}^{cr}$, by the definition of $\overline{g}^{cr}$ we have  
\begin{equation*}
L(E)\leqslant \overline{g}^{cr}\leqslant g, \ \text{for any}\ E\in \Sigma(0),
\end{equation*}
which means $\Sigma(0)\cap \mathcal{E}_{+}=\emptyset$ and thus $\Sigma(g)=\mathcal{E}_{0}$ by Theorem \ref{spec}. 
\end{proof}

\subsection{A question by Goldsheid-Khoruzhenko}

\begin{proof}[Proof of Corollary \ref{contour}]	
By Lemma \ref{analytic},  $\mathbb{R}\cap \mathcal{E}_{-}=\cup_{j}(a_{j},b_{j})$,   $\mathcal{E}_{0}=\cup_{j} \partial \Omega_{j}$, where $\partial\Omega_{j}:=\{E\in\mathcal{E}_{0}: a_{j}\leqslant \mathrm{Re}E \leqslant b_{j}\}$.  It follows from Lemma \ref{potential}\ref{item:growth} that $\mathcal{E}_{-}$ is bounded and hence $|a_{j}|,|b_{j}|<\infty$.  Hence $\mathcal{E}_{0}$ consists of a  family of contours.  Moreover, by the continuity of $L(E)$ we have $L(a_{j})=L(b_{j})=g$. 

We are going to show the number of contours is finite. Denote by $\Omega_{j}$ the domain surrounded by the contour $\partial \Omega_{j}$. We claim that
\begin{equation}\label{claim}
\Omega_{j}\cap\Sigma(0)\neq \emptyset, \ \text{for any} \ j.
\end{equation}
If not, then by Proposition \ref{equi} there exists $j$ such that
$(\alpha,S_{E})\in\mathcal{UH}$ for any $E\in\Omega_{j}$. Thus, by Lemma \ref{UHharmonic}, we have
\begin{equation*}
\Delta L(E)=0, \ \text{for any}\ E\in\Omega_{j}.
\end{equation*}
Since $\partial \Omega_{j}\subseteq\mathcal{E}_{0}$, we know $L(E)=g$ for $E\in \partial \Omega_{j}$. By the extremum principle of harmonic functions, we have $L(E)=g$ for any $E\in \Omega_{j}$. However, this contradicts to Lemma \ref{analytic}\ref{item:mono}.	

Assume that there exist infinite contours, i.e.  $\mathbb{R}\cap \mathcal{E}_{-}=\cup_{j=1}^{\infty}(a_{j},b_{j})$. We denote by $E^{*}$ an accumulation point of some subsequence $\{a_{i}\}_{i=1}^{\infty}$. Again, by the continuity of $L(E)$ and $\{a_{i}\}\subseteq \mathcal{E}_{0}$, we have
\begin{equation}\label{accu}
L(E^{*})=\lim_{i\rightarrow \infty} L(a_{i}) = g.
\end{equation}
On the other hand,  by the disjointness of these open intervals, one can conclude that for any $\delta>0$, there exists an interval $(a_{k},b_{k}) \subseteq \{z:|z-E^*|<\delta\}$. 

For any $\delta>0$, by \eqref{claim} there exists $E\in \Omega_{k} \cap \Sigma(0)\subseteq (a_{k},b_{k})$ such that $|E-E^*|<\delta$. However, by \eqref{accu} and $g>\overline{g}^{cr}$,
\begin{equation*}
|L(E^*)-L(E)|\geqslant L(E^*)-L(E)>g-\overline{g}^{cr}>0,
\end{equation*}
which means $L(\cdot)$ is discontinuous at $E^*$, resulting in a contradiction.
\end{proof}

\subsection{Two-dimensional spectrum}

Let us first prove Corollary \ref{single}.
\begin{proof}[Proof of Corollary \ref{single}]
	By Theorem 1.3 of \cite{WYZ}, we have 
	\begin{equation*}
		L(E)=\max\bigg\{\log\Big|\frac{E}{2}+\frac{\sqrt{E^2-4}}{2}\Big|,\log |\lambda|\bigg\}, \ \text{for any}\ E\in\mathbb{C}.
	\end{equation*} 
	Then the result follows directly from Theorem \ref{spec}. Indeed, if 
	$g=\log |\lambda|$, by the definition of $\mathcal{E}_{0}$ we can easily compute that 
	\begin{equation*}
		\mathcal{E}_{0}=\bigg\{E\in\mathbb{C}: \Big|\frac{E}{2}+\frac{\sqrt{E^2-4}}{2}\Big|\leqslant \mathrm{e}^{g}\bigg\},
	\end{equation*}
	which implies the result.
\end{proof}

To prove Corollary \ref{2d}, we need the following several results.
\begin{theorem}[\cite{GJYZ}, Theorem 5.1]\label{gne1}
Assume $v(x)=\sum_{|{ k}|\leqslant m} \widehat{v}_{ k} \mathrm{e}^{2\pi\mathrm{i} { k}x}\in C^{\omega}(\mathbb{T},\mathbb{R})$. For $\alpha\in\R\backslash\Q$, $E\in\mathbb{C}$, $y>0$, we have
\begin{equation*}
L(E,\mathrm{i}y)= L(E) -\sum_{\{j:\gamma_{j}(E)< 2\pi y\}}  \gamma_{j}(E)+2\pi\big(\#\{j:\gamma_{j}(E)<2\pi y\}\big)y.
\end{equation*}
\end{theorem}

\begin{theorem}[\cite{puig}, Main Application]\label{puig}
Assume $v(x)=\sum_{|{ k}|\leqslant m} \widehat{v}_{ k} \mathrm{e}^{2\pi\mathrm{i} { k}x}\in C^{\omega}(\mathbb{T},\mathbb{R})$. For $\alpha\in\R\backslash\Q$, $E\in\mathbb{C}$, we have
\begin{equation*}
L(E)=\sum\limits_{j=1}^{m} \gamma_{j}(E)+\log |\widehat{v}_{m}|.
\end{equation*}
\end{theorem}

\begin{theorem}[\cite{AJS14,BJ02, JKS, Powell}]\label{conLE}
For $(\alpha,A)\in \mathbb{R}\times C^{\omega}(\mathbb{T}, \mathrm{gl}(m,\mathbb{C}))$, then for any $j=1,\cdots,m$, $(\alpha, A)\mapsto \gamma_{j}(\alpha,A)\in [-\infty,\infty)$ are continuous at any $(\alpha, A)$  with $\alpha\in\mathbb{R}\backslash \mathbb{Q}$.
\end{theorem}

\begin{proof}[Proof of Corollary \ref{2d}]

By Theorem \ref{conLE}, the continuity of $\gamma_{m}(\cdot)$ implies that there exists $\delta$ such that 
\begin{equation*}
\gamma_{m}(E) < 2\pi{y}, \ \text{for all} \ |E|\leqslant \delta.
\end{equation*} 
By Theorem \ref{gne1} and Theorem \ref{puig},  we have
\begin{equation}\label{constant}
L(E,\mathrm{i}y)=\log|\widehat{v}_m|+2\pi my, \ \text{for all} \ |E|\leqslant \delta.
\end{equation} 
Consider the  Hatano-Nelson model with complex extension
\begin{equation*}
[\mathcal{H}(g,x+\mathrm{i}y)u]_{n}=-\mathrm{e}^{g}u_{n+1}-\mathrm{e}^{-g}u_{n-1}+v(x+n\alpha+\mathrm{i}y)u_{n},
\end{equation*} 
where $g=\log|\widehat{v}_m|+2\pi my$. We still denote by $\Sigma(g)$ the spectrum set of $\mathcal{H}(g,x+\mathrm{i}y)$ and let $\mathcal{E}_{0}=\{E:L(E,\mathrm{i}y)=g\}$.
By \eqref{constant} and Theorem \ref{spec}, $\{|E|\leqslant \delta\}\subseteq \mathcal{E}_{0}\subseteq\Sigma(g)$. 
\end{proof}

\section{Support of density of states measure}
In this section, we study the asymptotic distribution of eigenvalues of the truncated operators under periodic boundary conditions. First, we recall the following:
\begin{theorem}[\label{GKIDS}\cite{GK00}, Theorem 4.1, Corollary 4.6]
Let $(X,\mu,T)$ be ergodic and $X$ be a compact metric space. Let $v: X\rightarrow \mathbb{R}$ such that $\int_{X}\log (1+|v|)\mathrm{d}\mu<\infty$. Then
\begin{enumerate}[font=\normalfont, label={(\arabic*)}]
\item 	For $\mu$-a.e. $x\in X$ and $m$-a.e. $E\in\mathbb{C}$,
\begin{equation*}
	\frac{1}{n}\log |\det (\mathcal{H}_{n}(g,x)-E)|\rightarrow  L_{+}(E):=\max\{L(E),g\}.
\end{equation*}
	
\item \label{item:ids} For $\mu$-a.e. $x\in X$, 
\begin{equation*}
\mathrm{d}\mathcal{N}^{g,x}_{n}\rightharpoonup \mathrm{d}\mathcal{N}^{g} :=\frac{1}{2\pi} \Delta L_{+}(E) \mathrm{d}m, \ \text{as}\ n\rightarrow \infty.
\end{equation*}

\end{enumerate}
\end{theorem}
\begin{remark}
	In \cite{GK00}, except for considering the dynamical defined potentials $V(n)=v(T^{n}x)$, Goldsheid and Khoruzhenko proved Theorem \ref{GKIDS} in the general setting where $V(n)$ is a stationary ergodic sequence defined on a probability space.
\end{remark}

A direct corollary of Theorem \ref{GKIDS} is the Thouless formula of $\mathcal{H}(g,x)$.
\begin{corollary}\label{Cthouless}
	Under the assumption of Theorem \ref{GKIDS}, 
	For every $E\in \mathbb{C}$,
	\begin{equation*}
		L_{+}(E)= \int_{\mathbb{C}} \log |z-E|\mathrm{d}\mathcal{N}^{g}(z).
	\end{equation*}
\end{corollary}
\begin{proof}
	By the weak uniqueness of the subharmonic function, we only need to prove the Thouless formula for a.e. $E\in\mathbb{C}$. 
	Actually, by Theorem \ref{GKIDS}, one only needs to apply Lemma \ref{schwartz} with 
	\begin{equation*}
		p_{n}(E)=\frac{1}{n}\log |\det(\mathcal{H}_{n}(g,x)-E)|, \quad \mathrm{d}\mu_{n}=\mathrm{d}\mathcal{N}_{n}^{g,x}, \quad \mathrm{d}\mu=\mathrm{d}\mathcal{N}^{g}.
	\end{equation*}
	This finishes the proof.
\end{proof}
\begin{remark}
	For complex-valued analytic quasi-periodic potentials and $g=0$, the Thouless formula was proved in \cite{WWYZ}.
\end{remark}

Now we finish the proof of Theorem \ref{asy}.
\begin{proof}[Proof of Theorem \ref{asy}]

Note that according to the assumption on the continuity of the Lyapunov exponent,  $\mathcal{E}_{+}$ and $\mathcal{E}_{-}$ are open sets. Let us first prove 
\begin{equation*}
	\supp \mathrm{d}\mathcal{N}^{g}\subseteq\Sigma(g).
\end{equation*}
Suppose that $E\notin\Sigma(g)$, by Theorem \ref{spec} we have $E\in \mathcal{E}_{-} \cup (\mathcal{E}_{+}\backslash \Sigma(0))$. We need to consider two cases.

{\bf Case 1-1:} $E\in \mathcal{E}_{+}\backslash \Sigma(0)$. By Proposition \ref{equi} we have $(\alpha,S_{E})\in \mathcal{UH}$. By Lemma \ref{UHharmonic}, $L(\alpha,A)$ is pluri-harmonic w.r.t. the second variable for $(\alpha,A)\in \mathcal{UH}$. By the openness of $\mathcal{UH}$ and openness of $\mathcal{E}_{+}$, there exists $\delta=\delta(E,g)>0$ such that
\begin{equation*}
L(z)>g\ \text{and} \ \Delta L(z)=0,\ \text{for all}\ |z-E|<\delta.
\end{equation*}
By the definition of $L_{+}(\cdot)$, we have $L(z)=L_{+}(z)$ for any $z\in \mathcal{E}_{+}$ and thus
\begin{equation*}
\Delta L_{+}(z) = \Delta L(z) =0, \ \text{for all}\ |z-E|<\delta.
\end{equation*}
Thus it follows from Theorem \ref{GKIDS}\ref{item:ids} that
\begin{equation*}
\int_{|z-E|<\delta} \mathrm{d}\mathcal{N}^{g}(z) =\frac{1}{2\pi}\int_{|z-E|<\delta}\Delta L_{+}(z)\mathrm{d}m(z)=0,
\end{equation*}
which means $E\notin \supp \mathrm{d}\mathcal{N}^{g}$.

{\bf Case 1-2:} $E\in \mathcal{E}_{-}$. By the definition of $L_{+}(\cdot)$ and Corollary \ref{Cthouless},
\begin{equation*}
g=L_{+}(\tilde{E})=\int_{\mathbb{C}} \log|z-\tilde{E}|\mathrm{d}\mathcal{N}^{g}(z),\ \text{for all}\ \tilde{E}\in\mathcal{E}_{-}.
\end{equation*}
By the openness of $\mathcal{E}_{-}$, one can apply Lemma \ref{potential}(2) with
\begin{equation*}
\mathrm{d}\mu_{1}=\mathrm{d}\mathcal{N}^{g},\quad \mathrm{d}\mu_{2}=0,\quad h=g, \quad U=\mathcal{E}_{-},
\end{equation*}
then we have $\mathrm{d}\mathcal{N}^{g} |_{\mathcal{E}_{-}}=0$, which means $E\notin\supp \mathrm{d}\mathcal{N}^{g}$.

Finally, let us prove $\Sigma(g)\subseteq\supp \mathrm{d}\mathcal{N}^{g}$.
Suppose that $E\in \Sigma(g)$, by Theorem \ref{spec} we need to consider two cases.

{\bf Case 2-1:} $E\in \mathcal{E}_{+}\cap\Sigma(0)$. 
Since $v$ is real-valued, according to Avron-Simon \cite{AS83}, the Thouless formula for self-adjoint Schr\"odinger operators $H(x)$ reads 
\begin{equation}\label{Hthouless}
L(\tilde{E})=\int_{-\infty}^{\infty}\log |z-\tilde{E}|\mathrm{d}N(z)=\int_{\mathbb{C}}\log |z-\tilde{E}|\mathrm{d}\widehat{N}(z),\ \text{for all}\ \tilde{E}\in\mathbb{C},
\end{equation}
where $\mathrm{d}N$ is the density of states measure of $H$ and $\mathrm{d}\widehat{N}$ is defined as $\widehat{N}(Q):=N(Q\cap \mathbb{R})$ for any Borel set $Q\subseteq \mathbb{C}$. On the other hand, by the definition of $L_{+}(\cdot)$ and Corollary \ref{Cthouless},
\begin{equation}\label{Hp}
L(\tilde{E})=L_{+}(\tilde{E})=\int_{\mathbb{C}}\log |z-\tilde{E}|\mathrm{d}\mathcal{N}^{g}(z),\ \text{for all}\ \tilde{E}\in\mathcal{E}_{+}.
\end{equation}
Combining \eqref{Hthouless} with \eqref{Hp}, we have
\begin{equation*}
\int_{\mathbb{C}}\log |z-\tilde{E}|\mathrm{d}\mathcal{N}^{g}(z)=\int_{\mathbb{C}}\log |z-\tilde{E}|\mathrm{d}\widehat{N}(z),\ \text{for all}\ \tilde{E}\in\mathcal{E}_{+}.
\end{equation*}
Since $\mathcal{E}_{+}$ is an open set, we can now apply Lemma \ref{potential}\ref{item:equ} with
\begin{equation*}
\mathrm{d}\mu_{1}=\mathrm{d}\mathcal{N}^{g},\quad \mathrm{d}\mu_{2}=\mathrm{d}\widehat{N},\quad h=0,\quad U=\mathcal{E}_{+},
\end{equation*}
so that
\begin{equation}\label{dNN}
\mathrm{d}\widehat{N}|_{\mathcal{E}_{+}}=\mathrm{d}\mathcal{N}^{g}|_{\mathcal{E}_{+}}.
\end{equation}	
Recall that  $\Sigma(0)=\supp \mathrm{d}N $ \cite{AS83}.
Hence for $E\in \mathcal{E}_{+}\cap\Sigma(0)$, by \eqref{dNN},
\begin{equation*}
E\in \mathcal{E}_{+}\cap\supp \mathrm{d}N =\supp\mathrm{d}N|_{\mathcal{E}_{+}}=\supp \mathrm{d}\widehat{N}|_{\mathcal{E}_{+}}=\supp\mathrm{d}\mathcal{N}^{g}|_{\mathcal{E}_{+}} \subseteq \supp \mathrm{d}\mathcal{N}^{g}.
\end{equation*}

{\bf Case 2-2:} $E\in\mathcal{E}_{0}$. Without loss the generality, we assume $\mathrm{Im}E\geqslant 0$. By Lemma \ref{analytic}, we have $E+\mathrm{i}\eta\in \mathcal{E}_{+}$ for any $\eta>0$.
If $E\notin \supp \mathrm{d}\mathcal{N}^{g}$, then by Corollary \ref{Cthouless} and Lemma \ref{potential}\ref{item:subharm} there exists $\delta=\delta(E)>0$ such that
\begin{equation*}
\Delta L_{+}(z)=0,\ \text{for any}\ |z-E|<\delta.
\end{equation*}
Since $L_{+}(z)\geqslant g$ for any $|z-E|<\delta$ and $L_{+}(E)=g$, by the extremum principle of harmonic functions, we deduce that 
\begin{equation*}
L_{+}(z)\equiv g,\ \text{for any}\ |z-E|<\delta.
\end{equation*}
However, for any $\eta\in(0,\delta)$ we have
\begin{equation*}
L_{+}(E+\mathrm{i}\eta) = L(E+\mathrm{i}\eta)>g,
\end{equation*}
which is a contradiction.
\end{proof}

\appendix

\section{Weighted Hilbert space}\label{whs}

Let $g\geqslant 0$. Define 
\begin{equation*}
\ell^{2,g}(\mathbb{Z}):=\bigg\{u: \sum_{n} |u_{n}|^{2}\mathrm{e}^{2ng}<\infty\bigg\}.
\end{equation*}
It is easy to see that $\ell^{2,g}(\mathbb{Z})$ is a Hilbert space with the inner product $\langle u| v\rangle := \sum_{n} u_{n}\overline{v_{n}}\mathrm{e}^{2ng}$. In particular, the norm in $\ell^{2,g}(\mathbb{Z})$ is defined by $\|u\|_{2,g}= \sqrt{\langle u|u\rangle}$.

Let $V(\cdot): \mathbb{Z}\rightarrow \mathbb{C}$. We define an operator $\widetilde{\mathcal{H}}(g): \ell^{2,g}(\mathbb{Z})\rightarrow \ell^{2,g}(\mathbb{Z})$ having the same form as $\mathcal{H}(g)$ in \eqref{HNV}. Specifically, for any $u\in\ell^{2,g}(\mathbb{Z})$,
\begin{equation*}
[\widetilde{\mathcal{H}}(g)u]_{n}= -\mathrm{e}^{g}u_{n+1}-\mathrm{e}^{-g}u_{n-1}+ V(n)u_{n},\ n\in\mathbb{Z}.
\end{equation*}
Denote by $H:=\widetilde{\mathcal{H}}(0)$ for the usual Schr\"odinger operator on $\ell^{2}(\mathbb{Z})$.

\begin{lemma}\label{unitary}
Let $[Wu]_{n}=\mathrm{e}^{-ng}u_{n}$. Then $W:\ell^{2}(\mathbb{Z}) \rightarrow \ell^{2,g}(\mathbb{Z})$ is unitary. Moreover, $\widetilde{\mathcal{H}}(g)$ is unitarily equivalent to $H$.
\end{lemma}
\begin{proof}
By the direct calculation, for any $u\in \ell^{2}(\mathbb{Z})$,
\begin{equation*}
\|Wu\|_{2,g}^{2} = \langle Wu|Wu\rangle=\sum_{n} |\mathrm{e}^{-ng} u_{n}|^{2}\mathrm{e}^{2ng} = \|u\|_{2}^{2},
\end{equation*}
thus $W$ preserves the norm. It is obvious that $W$ is bijective, hence $W$ is unitary. Moreover, one can get that $W^{-1}\widetilde{\mathcal{H}}(g)W=H$.
\end{proof}

For $V(\cdot):\mathbb{Z}\rightarrow \mathbb{R}$, the following result implies $\widetilde{\mathcal{H}}(g)$ is indeed self-adjoint on $ \ell^{2,g}(\mathbb{Z})$ even though $\mathcal{H}(g)$ remains non-self-adjoint on $\ell^{2}(\mathbb{Z})$ for any $g>0$.

\begin{lemma}\label{sa}
Let $V:\mathbb{Z} \rightarrow\mathbb{R}$ and $g> 0$. Then $\widetilde{\mathcal{H}}(g): \ell^{2,g}(\mathbb{Z})\rightarrow \ell^{2,g}(\mathbb{Z})$ is self-adjoint.
\end{lemma}
\begin{proof}
Since the Schr\"odinger operator $H:\ell^{2}(\mathbb{Z})\rightarrow \ell^{2}(\mathbb{Z})$ is self-adjoint, the proof follows from Lemma \ref{unitary}.
\end{proof}

Let $V(n)=v(T^{n}x)$ with $T:X\rightarrow X$ ergodic and $X$ compact. Then the dynamical defined Hatano-Nelson model $\widetilde{\mathcal{H}}(g,x): \ell^{2,g}(\mathbb{Z})\rightarrow\ell^{2,g}(\mathbb{Z})$ is defined as
\begin{equation*}
[\widetilde{\mathcal{H}}(g,x)u]_{n}=-\mathrm{e}^{g}u_{n+1}-\mathrm{e}^{-g}u_{n-1}+v(T^{n}x)u_{n}.
\end{equation*}
The spectrum of $\widetilde{\mathcal{H}}(g,x)$, denoted by $\widetilde{\Sigma}(g)$, is independent of a.e. $x\in X$ by the ergodicity of $T$. We also consider the finite size operators with Dirichlet boundary conditions,
\begin{equation*}
\widetilde{\mathcal{H}}_{n}(g,x):=\mathcal{H}_{n}(g,x):=\begin{pmatrix}
v(T^{0}x)&-\mathrm{e}^{g}&0&\cdots&0\\
-\mathrm{e}^{-g}&v(T^{2}x)&-\mathrm{e}^{g}&\ddots&0\\
\vdots&\ddots&\ddots&\ddots&\vdots\\
0&\ddots&-\mathrm{e}^{-g}&v(T^{n-2}x)&-\mathrm{e}^{g}\\
0&\cdots&0&-\mathrm{e}^{-g}&v(T^{n-1}x)
\end{pmatrix}.
\end{equation*}
Let $\mathrm{d}\widetilde{\mathcal{N}}_{n}^{g,x}=\frac{1}{n}\sum_{j=1}^{n} \delta_{\widetilde{E}_{j}(g,x)}$ where $\{\widetilde{E}_{j}(g,x)\}_{1\leqslant j\leqslant n}$ are the eigenvalues of $\widetilde{\mathcal{H}}_{n}(g,x)$. The following results imply that $\widetilde{\Sigma}(g)$ can be approximated by the spectrum of the finite size operators with Dirichlet boundary conditions.

\begin{proposition}\label{spectrum1}
Let $g> 0$, $X$ compact, $v\in C^{0}(X,\mathbb{R})$ and $T: X \rightarrow X$ ergodic.
\begin{enumerate}[font=\normalfont, label={(\arabic*)}]
\item For any $x\in X$, 
\begin{equation*}
\mathrm{d}\widetilde{\mathcal{N}}^{g,x}_{n}\rightharpoonup \mathrm{d}\widetilde{\mathcal{N}}^{g} =\frac{1}{2\pi} \Delta L(E) \mathrm{d}m, \ \text{as}\ n\rightarrow \infty.
\end{equation*}

\item Moreover,
\begin{equation*}
\supp \mathrm{d}\widetilde{\mathcal{N}}^{g}=\widetilde{\Sigma}(g)=\Sigma(0).
\end{equation*}

\end{enumerate}	
\end{proposition}

\begin{proof}[Proof of Proposition \ref{spectrum1}]
Let $W_{n}=\mathrm{diag} (\mathrm{e}^{-g},\cdots,\mathrm{e}^{-ng})$. It is obvious that 
\begin{equation*}
W_{n}^{-1}\widetilde{\mathcal{H}}_{n}(g,x) W_{n}=H_{n}(x),
\end{equation*}
where $H_{n}(x):=\mathcal{H}_{n}(0,x)$ is the truncated Schr\"odinger operator. The proof follows from \cite{AS83} and Lemma \ref{sa}.
\end{proof}

\section*{Acknowledgements}
X. Wang would like to thank Zhenghe Zhang and David Damanik for useful discussions on Green's function. This work was partially supported by National Key R\&D Program of China (2020YFA0713300) and Nankai Zhide Foundation.  J. You was also partially supported by NSFC grant (11871286). Q. Zhou was supported by NSFC grant (12071232).


\begin{thebibliography}{99}
\bibitem{AHN}A. Amir, N. Hatano, D. R. Nelson: Non-Hermitian localization in biological networks, {\it Phys. Rev. E }, \textbf{93}(4), 042310 (2016).


\bibitem{av11}A. Avila: Density of positive Lyapunov exponent for $\mathrm{SL}(2,\mathbb{R})$-cocycles, {\it J. Am. Math. Soc.}, \textbf{24}(4), 999-1014 (2011).

\bibitem{AEV}A. Avila, A. Eskin, M. Viana: Continuity of the Lyapunov exponents of random matrix products, arXiv: 2305.06009 (2023).


\bibitem{AJS14}A. Avila, S. Jitomirskaya, C. Sadel: Complex one-frequency cocycles,  {\it J. Eur. Math. Soc.}, \textbf{16}(9), 1915-1935 (2014).

\bibitem{Av0}A. Avila: Global theory of one-frequency Schr\"odinger operators, {\it Acta Math.}, \textbf{215}(1), 1-54 (2015).

\bibitem{AS83}J. Avron, B. Simon: Almost periodic Schr\"odinger operators II. The integrated density of states, {\it Duke Math. J.}, \textbf{50}(1), 369-391 (1983).

\bibitem{BLLS}S. Beckus, D. Lenz, M. Lindner, C. Seifert: Note on spectra of non-self-adjoint operators over dynamical systems, {\it Proc. Edinburgh Math. Soc.}, \textbf{61}(2), 371-386 (2018).


\bibitem{bv}C. Bocker and M. Viana: Continuity of Lyapunov exponents for random two-dimensional matrices, {\it Ergodic Theory \& Dynam. Systems}, \textbf{37}, 1413-1442 (2017).


\bibitem{BJ02}J. Bourgain, S. Jitomirskaya: Continuity of the Lyapunov exponent for quasiperiodic operators with analytic potential, {\it J. Statist. Phys.}, \textbf{108}(5/6), 1203-1218 (2002).


\bibitem{Co20}M. Colbrook: Pseudoergodic operators and periodic boundary conditions, {\it Math. Comput.}, \textbf{89}(322), 737-766 (2020).


\bibitem{DNS}K. A. Dahmen, D. R. Nelson, N. M. Shnerb: Population dynamics and non-Hermitian localization, Statistical mechanics of biocomplexity, Springer, 124-151 (1999).

\bibitem{DFbook}D. Damanik, J. Fillman: One-dimensional Schr\"odinger operators, I, General theory, Graduate Studies in Mathematics, 221, American Mathematical Society.

\bibitem{Dav01A}E. B. Davies: Spectral properties of random non-self-adjoint matrices and operators, {\it Proc. R. Soc. Lond. A}, \textbf{457}(2005), 191-206 (2001).

\bibitem{Dav01}E. B. Davies: Spectral theory of pesudo-ergodic operators, {\it Commun. Math. Phys.}, \textbf{216}(3), 687-704 (2001). 



\bibitem{Dav02}E. B. Davies: Non-self-adjoint differential operators, {\it Bull. London Math. Soc.}, \textbf{34}(5), 513-532 (2002).


\bibitem{Fur97}A. Furman: On the multiplicative ergodic theorem for the uniquely ergodic systems, {\it Ann. Inst. Henri Poincar\'{e}}, \textbf{33}(6), 797-815 (1997).


\bibitem{GJYZ}L. Ge, S. Jitomirskaya, J. You, Q. Zhou: Multiplicative Jensen's formula and quantitative global theory of one-frequency Schr\"odinger operators, arXiv: 2306.16387 (2023).

\bibitem{GK00}I. Goldsheid, B. Khoruzhenko: Eigenvalue curves of asymmetric tridiagonal matrices, {\it Electron. J. Probab.}, \textbf{5}, 1-28 (2000).

\bibitem{GK03}I. Goldsheid, B. Khoruzhenko: Regular spacings of complex eigenvalues in the one-dimensional non-Hermitian Anderson model, {\it Commun. Math. Phys.}, \textbf{238}(3), 505-524 (2003).

\bibitem{GS18}I. Goldsheid, S. Sodin: Real eigenvalues in the non-Hermitian Anderson model, {\it Ann. Appl. Probab.}, \textbf{28}(5), 3075-3093 (2018).


\bibitem{Gong}Z. Gong, Y. Ashida, K. Kawabata, K. Takasan, S. Higashikawa, M. Ueda: Topological phases of non-Hermitian systems, {\it Phys. Rev. X}, \textbf{8}(3), 031079 (2018).


\bibitem{puig} A. Haro, J. Puig: A Thouless formula and Aubry duality for long-range Schr\"odinger skew-products, {\it Nonlineraity}, \textbf{26}(5), 1163-1187(2013).

\bibitem{HN96}N. Hatano, D. R. Nelson: Localization transitions in non-Hermitian quantum mechanics, {\it Phys. Rev. Lett.}, \textbf{77}(3), 570-573 (1996).

\bibitem{HN97}N. Hatano, D. R. Nelson: Vortex pinning and non-Hermitian quantum mechanics, {\it Phys. Rev. B}, \textbf{56}(14), 8651-8673 (1997).

\bibitem{HN98}N. Hatano, D. R. Nelson: Non-Hermitian delocalization and eigenfunctions, {\it Phys. Rev. B}, \textbf{58}(13), 8384-8390 (1998).

\bibitem{JKS}S. Jitomirskaya, D. A. Koslover, M. S. Schulteis: Continuity of the Lyapunov exponent for analytic quasiperiodic cocycles, {\it Ergod. Th. \& Dynam. Sys.}, \textbf{29}(6), 1881-1905 (2009).

\bibitem{John82}R. Johnson: The recurrent Hill’s equation, {\it J. Differ. Equ.}, \textbf{46}, 165-193 (1982).

\bibitem{John86}R. Johnson: Exponential dichotomy, rotation number, and linear differential operators with bounded coefficients, {\it J. Differ. Equ.} \textbf{61}(1), 54-78 (1986).

\bibitem{JPS87}R. Johnson, K. Palmer, G. Sell: Ergodic properties of linear dynamical systems, {\it SIAM J. Math. Anal.},  \textbf{18}(1), 1-33 (1987).


\bibitem{King73}J. Kingman: Subadditive ergodic theory, {\it Ann. Probab.}, \textbf{1}(6), 883-899 (1973). 


\bibitem{Lan}N. S. Landkof: Foundations of modern potential theory, Vol. 180, Springer, Berlin, (1972).


\bibitem{longhi}S. Longhi: Topological phase transition in non-Hermitian quasicrystals, {\it Phys. Rev. Lett.}, \textbf{122}(23), 237601 (2019).


\bibitem{NS98}D. R. Nelson, N. M. Shnerb: Non-Hermitian localization and population biology, {\it Phys. Rev. E}, \textbf{58}(2), 1383-1403 (1998).

\bibitem{PF}L. A. Pastur, A.L. Figotin: Spectra of random and almost-periodic operators, Springer, Berlin, (1992).

\bibitem{Powell}M. Powell: Continuity of the Lyapunov exponent for analytic multi-frequency quasiperiodic cocycles, arXiv: 2210.09285 (2022).

\bibitem{Ransford}T. Ransford: Potential theory in the complex plane, {\it London Mathematical Society Student Texts}, \textbf{28}, Cambridge University press, (1995).


\bibitem{SS76}R. Sacker, G. Sell: Existence of dichotomies and invariant splittings for linear differential systems, II, {\it J. Differ. Equ.}, \textbf{22}(2), 478-496 (1976).


\bibitem{viana}M. Viana: Lecture on Lyapunov exponents, Cambridge Studies in Advanced Mathematics 145, Cambridge University Press (2014).

\bibitem{WWYZ}X. Wang, Z. Wang, J. You, Q. Zhou: Winding number, density of states and acceleration, arXiv: 2304.02486 (2023).

\bibitem{ZFW}Z. Wang, J. You, Q. Zhou: Non-self-adjoint quasi-periodic operators with complex spectrum, arXiv: 2306.04457 (2023).

\bibitem{WYZ}X. Wang, J. You, Q. Zhou: Isospectrum of non-self-adjoint almost-periodic Schr\"odinger operators, arXiv: 2206.02400 (2022).

\bibitem{WY}Y. Wang, J. You: Examples of discontinuity of Lyapunov exponent in smooth quasiperiodic cocycles, {\it Duke Math. J.} , \textbf{13}, 2363-2412 (2013).

\end{thebibliography}
\end{document}